\documentclass[12pt,reqno]{amsart}
\usepackage{amsmath}
\usepackage{amsfonts}
\usepackage{amssymb}
\usepackage{amsthm}
\usepackage[alphabetic]{amsrefs} % biblio
\usepackage{algorithm,setspace}
\usepackage{algpseudocode}
\usepackage{booktabs}
\usepackage{mathrsfs}
\usepackage{mdframed}
\usepackage{caption}
\usepackage[most]{tcolorbox}
\usepackage[makeroom]{cancel}
\usepackage{array}
\usepackage[shortlabels]{enumitem}
\usepackage[hidelinks]{hyperref}
\hypersetup{
    colorlinks=true,
    linkcolor=magenta,      % Color for \cref, \ref, \eqref, etc.
    citecolor=magenta,     % Color for citations
    filecolor=magenta,
    urlcolor=magenta,
}
\usepackage[noabbrev,nosort,capitalise]{cleveref}

\usepackage[margin=1in,headsep=30pt]{geometry}
\pdfpagewidth=\paperwidth
\usepackage{xcolor}
\usepackage{color,soul}
\usepackage{tikz}
\usepackage{tikz-cd}
\usepackage{tkz-euclide}
\usepackage[latin1]{inputenc}
\usepackage{fancyhdr}
\usepackage{stmaryrd}
\usetikzlibrary{
  graphs,
  graphs.standard
}
\usepackage{titlesec}

\titleformat{\section}
  {\normalfont\Large\bfseries}{\thesection }{1em}{}
\titleformat{\subsection}
  {\normalfont\large\bfseries}{\thesubsection.}{1em}{}
\titleformat{\subsubsection}
  {\normalfont\normalsize\itshape}{\thesubsubsection.}{1em}{}
\newcolumntype{C}{>{$}c<{$}}

\newmdenv[
  leftline=true,
  rightline=false,
  topline=false,
  bottomline=false,
  linewidth=2pt,
  linecolor=gray!50,
  skipabove=\baselineskip,
  skipbelow=0pt,
  innerleftmargin=10pt,
  innerrightmargin=5pt,
  innertopmargin=5pt,
  innerbottommargin=5pt
]{leftbar}

\newcommand{\arxiv}[1]
{\bgroup\color{blue!50}%
\href{http://arxiv.org/abs/#1}{\texttt{\scriptsize arXiv:#1}}\egroup}

\theoremstyle{plain}
\newtheorem{thm}{Theorem}
\newtheorem{cor}[thm]{Corollary}

\newtheorem{prop}[thm]{Proposition}

\newtheorem{defn}[thm]{Definition}
\newtheorem{eg}{Example}

\newtheorem{rem}[thm]{Remark}

\numberwithin{equation}{section}
\fancypagestyle{Euler}{%
  \fancyhf{} % clear all fields
  
  \fancyhead[LO]{\sl Elliptic Functions}
  \fancyhead[RE]{Shihan Kanungo}
  \fancyhead[LE,RO]{\thepage}
}%
\newtheorem{theorem}{Theorem}[section]
\newtheorem{lemma}{Lemma}

\allowdisplaybreaks

\begin{document}

\newpage
\thispagestyle{empty}

\newpage 
\setlength{\parskip}{5pt}
\thispagestyle{empty}
\setcounter{page}{1}

\begin{center}
    {\LARGE \bf Elliptic Functions and Eisenstein Series} 
    \vskip 5pt
    {\normalsize From Complex Lattices to Modular Forms}
    
    \vspace{0.3in}
    {\large  Shihan Kanungo} 
    
    \vspace{0.05in}
    {\small San Jos\'e State University}
\end{center}
\date{\today}

\vspace{0.2in}
\begin{center}
    {\bf Abstract}
\end{center}
\begin{quote}
    {\footnotesize In this expository paper, we provide an introduction to elliptic functions and Eisenstein series from a classical analytic perspective while highlighting their broader mathematical significance. The exposition assumes familiarity with basic complex analysis, particularly meromorphic functions, contour integration, and power series expansions.}
\end{quote} 

\setcounter{section}{0}

\vspace{5mm}
\section{Introduction}

Elliptic functions occupy a remarkable position at the crossroads of complex analysis, algebra, number theory, and geometry. Historically, they arose from attempts to understand certain definite integrals appearing in the study of arc lengths of ellipses, much as trigonometric functions arise naturally from integrals associated with the circle. Yet the eventual theory developed by nineteenth-century mathematicians revealed something far richer: elliptic functions are not merely analogues of trigonometric functions, but the first substantial examples of genuinely doubly periodic meromorphic functions on the complex plane. Their study led naturally to the theory of lattices, modular forms, elliptic curves, and ultimately to major developments in modern number theory.

The simplest periodic functions in analysis are the trigonometric functions, which repeat their values after translation by a single period.  Elliptic functions generalize this phenomenon by possessing two independent periods. The geometry of these two-dimensional period lattices gives the subject much of its richness. In modern language, elliptic functions may be viewed as meromorphic functions on complex tori; thus the theory naturally connects complex analysis with geometry and topology.

A central role in this theory is played by the Eisenstein series. These are infinite sums attached to lattices which encode subtle information about the geometry of the lattice. Although initially appearing as auxiliary analytic objects, Eisenstein series turn out to be fundamental building blocks in the theory of elliptic functions and modular forms. In particular, the celebrated Weierstrass $\wp$-function can be expressed in terms of Eisenstein series, and the coefficients appearing in its differential equation are themselves lattice invariants built from these series. From a modern perspective, Eisenstein series provide one of the simplest and most explicit examples of modular forms, objects that now permeate large areas of mathematics ranging from arithmetic geometry to mathematical physics.

The development of the subject spans much of nineteenth-century mathematics. The origins lie in the work of mathematicians such as Leonhard Euler, Adrien-Marie Legendre, and above all Niels Henrik Abel and Carl Gustav Jacob Jacobi, who transformed the study of elliptic integrals into the theory of elliptic functions. Abel recognized the deep algebraic structure underlying these functions, while Jacobi introduced theta functions and established many of the remarkable identities characteristic of the theory. Later, Karl Weierstrass developed a systematic analytic framework centered on the $\wp$-function, placing the subject on rigorous foundations and revealing its intrinsic relation to doubly periodic meromorphic functions.

The modern theory of modular forms emerged from these ideas. The transformation properties of Eisenstein series under changes of lattice revealed unexpected symmetries encoded by the modular group. In the twentieth century, these ideas became central to number theory through the work of mathematicians such as Erich Hecke, Andr\'e Weil, and Jean-Pierre Serre. Today, elliptic functions and Eisenstein series appear throughout mathematics, including the theory of elliptic curves, the proof of Fermat's Last Theorem, representation theory, and string theory.

The aim of this article is to provide an introduction to elliptic functions and Eisenstein series from a classical analytic perspective while highlighting their broader mathematical significance. The exposition assumes familiarity with basic complex analysis, particularly meromorphic functions, contour integration, and power series expansions.

We begin by introducing elliptic functions, and we prove structural results about their poles and zeros. Next, we introduce the Weierstrass $\wp$-function and study its convergence, periodicity, and poles. 

The final part of the article develops Eisenstein series in detail. We then show how these series arise naturally in the Laurent expansion of the $\wp$-function and use them to prove a fundamental differential equation for $\wp$. Using this, we prove the addition formula for $\wp$, as well as some other interesting facts. We conclude by discussing the modularity of the Eisenstein series, and how they connect to elliptic curves.

The theory of elliptic functions is often striking for the way seemingly disparate ideas converge within it. Infinite series, algebraic equations, geometric quotients, and analytic continuation all appear naturally and reinforce one another. Eisenstein series, in particular, illustrate how highly symmetric analytic constructions can encode profound arithmetic information. Although the subject originated in concrete computational problems involving integrals, it ultimately became one of the foundational pillars of modern mathematics.
\section{Elliptic Functions}
In real analysis, periodic functions play a key role, particularly in applications like Fourier analysis. The complex-analytic analogues of periodic functions are \textit{elliptic functions}. 

Roughly speaking, an elliptic function is a meromorphic function that is periodic ``in two directions.'' The formal definition is as follows.

\begin{defn}\label{defn: elliptic functions}
    Given $\omega_1,\omega_2\in \mathbb C$ with $\omega_1/\omega_2\not\in \mathbb R$, a meromorphic function $f$ on $\mathbb C$ is called an \textit{elliptic function} if
    \[f(z+\omega_1) = f(z+\omega_2) = f(z)\]
    for all $z\in \mathbb C$.
\end{defn}
More precisely, we let $\Lambda = \{a\omega_1 + b\omega_2: a,b\in \mathbb Z\}$ be the \textit{lattice generated by $\omega_1$, $\omega_2$}. We then say that $f$ is an \textit{elliptic function on $\Lambda$}. 

Due to the periodicity of $f$, it suffices to define $f$ on the parallelogram in $\mathbb C$ with vertices $0$, $\omega_1$, $\omega_2$, and $\omega_1 + \omega_2$. This parallelogram is called the \textit{fundamental parallelogram $P$}.

    \begin{center}
    \includegraphics[width = 0.4\textwidth]{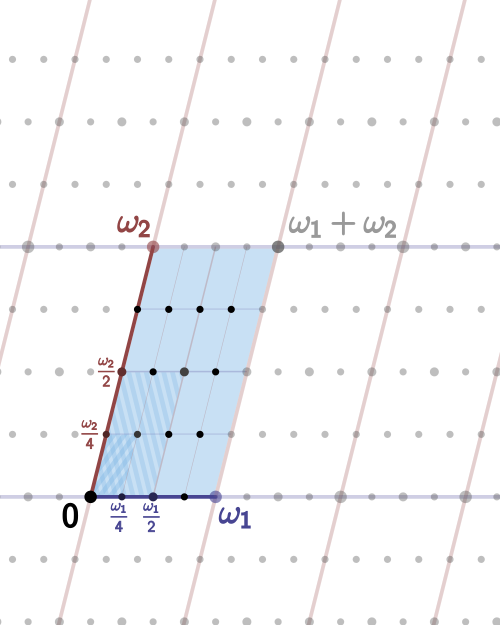}

    \textbf{Figure 1.} A fundamental parallelogram.
\end{center}

\vskip 5pt

Gluing the top and bottom edges and the left and right edges of $P$ together produces a torus, and hence we can think of elliptic functions as ``meromorphic functions on a torus.''

\begin{leftbar}
\begin{rem}
    The condition that $\omega_1/\omega_2\not\in \mathbb R$ is important. Indeed, if $\omega_1/\omega_2\in \mathbb Q$, then we can find a single period $\omega$ of $f$ such that $\omega_1$ and $\omega_2$ are both multiples of $\omega$. If $\omega_1/\omega_2$ is an irrational real number, then there exist linear combinations of $\omega_1$ and $\omega_2$ that are arbitrarily close to zero. By periodicity of $f$, this implies that $f'(z) = 0$ for all $z$, so $f$ is constant. Hence in both cases, we do not get the situation we want.
\end{rem} 
\end{leftbar}

The reader might ask why we specified that an elliptic function is a meromorphic function, rather than being an entire function. This is because of the following result.

\begin{theorem}
    The only entire elliptic functions are constant functions.
\end{theorem}
\begin{proof}
    Let $f$ be an entire elliptic function. Since the fundamental parallelogram $P$ is compact, $f(P)$ is bounded. Since every point in $\mathbb C$ can be translated to $P$ using periodicity, it follows that $f(\mathbb C) = f(P)$ is bounded. But the only bounded entire functions are constant, so we conclude.
\end{proof}

So any interesting function must have poles. The next simplest possibility is that there is exactly one simple pole per fundamental parallelogram, but unfortunately this too is impossible. This will be a consequence of the following result.

\begin{prop}
    The sum of the residues in a fundamental parallelogram of an elliptic function $f$ is zero.
\end{prop}
\begin{proof}
    Let $\Gamma$ be the boundary of a fundamental parallelogram. Then by the Residue theorem,
    \[\frac{1}{2\pi i} \int_\Gamma f(z)\, dz\]
    is the sum of the residues inside $\Gamma$. On the other hand, the integrals of $f(z)$ along the top and bottom sides of $\Gamma$ cancel. Indeed, periodicity implies that the values are the same, and we are integrating in opposite directions. Similarly, the integrals of $f$ over the left and right sides cancel, so the whole integral equals zero. This completes the proof. 
\end{proof}

\begin{cor}
    There are no elliptic functions with exactly one pole (counted with multiplicity) per fundamental parallelogram.
\end{cor}
\begin{proof}
    If there is only one simple pole inside $\Gamma$, then it has a nonzero redsidue, so the sum of the residues cannot be zero.
\end{proof}

We conclude this section with some more results of the same flavor. Notice that all these results include the phrase ``inside a fundamental parallelogram''. This is for good reason, because considering the whole complex plane doesn't give us any extra information.

\begin{theorem}\label{thm: count zeros and poles}
    Inside a fundamental parallogram, an elliptic function $f$ has the same number of zeros as poles (counted with multiplicity).
\end{theorem}
\begin{proof}
To prove this, we use the argument principle:
\begin{prop}[Argument Principle]
    Let $\Gamma$ be a contour, and let $f$ be meromorphic on $\Gamma$ and its interior. Let $n_0(f)$ and $n_{\infty}$ be the number of zeros and poles of $f$ inside $\Gamma$, respectively. Then
    \[\frac{1}{2\pi i}\int_\Gamma \frac{f'(z)}{f(z)}\, dz = n_0(f) - n_{\infty}(f).\]
\end{prop}
To apply this result, set $\Gamma$ to be the fundamental parallelogram. Then $\frac{f'(z)}{f(z)}$ is also elliptic, which means that the integral on the left vanishes. Hence $n_0(f) = n_\infty (f)$, proving the claim.
\end{proof}

\begin{prop}\label{prop: sum of roots}
Let $f$ be an elliptic function. Let $a_1,\dots, a_n$ be the zeros and poles of $f$ inside a fundamental parallelogram, and let $m_1,\dots, m_n$ be the corresponding multiplicities (poles have negative multiplicities). Then
\[\sum_{k=1}^n a_im_i \in \Lambda.\]
\end{prop}
\begin{proof}
    As before, we consider an integral around the boundary $\Gamma$ of the fundamental parallelogram. Consider the integral
\[\frac{1}{2\pi i}\int_\Gamma z\frac{f'(z)}{f(z)}\, dz.\]
By the argument principle, the residue of $\frac{f'(z)}{f(z)}$ at $a_i$ is $m_i$. It follows that \[\mathrm{Res}\left(\frac{zf'(z)}{f(z)}; a_i\right)=a_km_k.\] Thus, the Residue theorem implies that \[\frac{1}{2\pi i}\int_\Gamma z\frac{f'(z)}{f(z)}\, dz = \sum_{k=1}^n a_km_k.\]

Now we compute the integral differently. First we evaluate the sum of the integrals along the top and bottom sides. Since
\[(z+\omega_2) \frac{f'(z+\omega_2)}{f(z+\omega_2)}= z\frac{f'(z)}{f(z)}+\omega_2\frac{f'(z)}{f(z)},\]
the $z\frac{f'(z)}{f(z)}$ term cancels, and we are left with
\[\frac{\omega_2}{2\pi i}\int_{0}^{\omega_1} \frac{f'(z)}{f(z)}\, dz.\tag{$1$}\]
The anti-derivative of $\frac{f'(z)}{f(z)}$ is $\log f(z)$. So the integral is $\log f(\omega_1) - \log f(0)$. However, we have to be careful of branch cuts. As we go from $0$ to $\omega_1$, we might move to a different branch of $\log$, so the value can be any multiple of $2 \pi i$. Hence $(1)$ can be any multiple of $\omega_1$. Similarly, the left and right sides contribute a multiple of $\omega_2$, so the whole expression can be any element of $\Lambda$. It follows that
\[\sum_{k=1}^n a_k m_k \in \Lambda,\]and we are done.
\end{proof}

\section{The Weierstrass $\wp$-function}
Now that we have seen many properties of elliptic functions, it is time to construct one.

A possible first attempt to construct an elliptic function would be to write down an expression like
\[f(z) = \sum_{\omega\in \Lambda} \frac{1}{(z+\omega)^2}.\]
This function is automatically periodic, because adding any element of $\Lambda$ does not change the set of terms being summed. However, there is one problem: this series doesn't converge. Indeed, we can write 
\[\sum_{\omega\in \Lambda} \frac{1}{|z+\omega|^2} = \sum_{n=0}^\infty \sum_{|a|+|b| = n} \frac{1}{|z + a\omega_1 + b\omega_2|^2}.\tag{$2$}\]
Now 
\[|z+a\omega_1 + b\omega_2| \le |z| + |a|\cdot |\omega_1| + |b| \cdot |\omega_2|.\]
In particular, this is bounded above by $Cn$ for some constant $C>0$. Hence, $(2)$ dominates
\[\sum_{n=0}^\infty \sum_{|a| + |b| = n}\frac{1}{C^2n^2} = \sum_{n=0}^\infty \frac{4}{C^2n}.\]
But this sum diverges, so the original sum diverges as well. This means that the series defining $f(z)$ is not absolutely convergent.

This problem can be fixed in multiple ways. One way would be to replace $\frac{1}{(z+\omega)^2}$ with a different function in $z+\omega$; periodicity still holds, for the same reasons. For example, we could consider
\[\sum_{\omega\in \Lambda} \frac{1}{(z+\omega)^3}.\tag{3}\]
As we shall see soon, this is in fact an important elliptic function.

However, we will take a different approach, which will be to ``force'' the series \[\sum_{\omega\in \Lambda}\frac{1}{(z+\omega)^2}\] to converge, by subtracting a constant from each term. Here is the definition.

\begin{defn}[Weierstrass $\wp$-function]\label{defn: wp}
    The \textit{Weierstrass $\wp$-function} is defined by
        \[\wp(z) = \frac{1}{z^2}+\sum_{\omega \in \Lambda\setminus\{0\}}\left(\frac{1}{(z+\omega)^2}-\frac{1}{\omega^2}\right).\]
\end{defn}
Since we have broken the symmetry in the definition of $\wp(z)$, periodicity now does not come for granted. Fortunately, it still holds.

\begin{theorem}
    The $\wp$-function is an elliptic function, with a double pole at each lattice point $\omega\in \Lambda$, and no poles anywhere else. Furthermore, it is an even function.
\end{theorem}
\begin{proof}
    Note that as $\omega\to\infty$, 
    \[\frac{1}{(z+\omega)^2}-\frac{1}{\omega^2} \sim \frac{1}{\omega^3}.\]
    Since a two-dimensional sum of $\frac{1}{\omega^3}$ converges, the series defining $\wp(z)$ occurs unless there is a division by zero. This happens precisely when $z\in \Lambda$, in which case we have a double pole. That $\wp$ is an even function follows from the definition and the fact that $\omega\in \Lambda$ implies $-\omega\in \Lambda$.

    Hence, we just need to prove periodicity. To do this, we observe that
    \[\wp'(z) = -2 \sum_{\omega}\frac{1}{(z+\omega)^3}\]
    (this is why (3) is important!). This function is periodic with periods $\omega_1$ and $\omega_2$, so
    \[\frac{d}{dz}(\wp(z+\omega_1)-\wp(z)) = \frac{d}{dz}(\wp(z+\omega_2)-\wp(z)) = 0.\]
    Hence, for some constants $\alpha$, $\beta$, we have \[\wp(z+\omega_1)-\wp(z) = \alpha\quad \text{and} \quad \wp(z+\omega_2)-\wp(z)=\beta\]
    To determine the constants, observe that since $\wp$ is an even function, \[\wp(\omega_1/2) - \wp(-\omega_1/2) = \wp(\omega_2/2) - \wp(-\omega_2/2) =  0.\]
    Hence $\alpha= \beta = 0$, and we conclude.
\end{proof}
We would like to understand the function $\wp(z)$ a little bit better.

The first step to do that would be to understand its derivative $\wp'(z)$. Immediately, we see that $\wp'(z)$ has a pole of order $3$ at $0$, and no poles anywhere else. Consequently, \cref{thm: count zeros and poles} implies that $\wp'(z)$ has three zeros, counted with multiplicity\footnote{In the fundamental parallelogram, of course. We will henceforth omit this clause when the context is clear.}. What might these zeros be? I claim that they are
\[\frac{\omega_1}{2}, \quad \frac{\omega_2}{2}, \quad \text{and}\quad \frac{\omega_1+\omega_2}{2}.\]
The motivation for these particular numbers is that $\wp(z)$ is an even function. Since $\wp(z)$ is even, $\wp'(z)$ is odd, i.e. $\wp'(z) + \wp'(-z) = 0$. Each of the numbers above satisfies $z - (-z)\in \Lambda$ (and up to periodicity, they are the only non-lattice points to do so). Hence, they all satisfy
\[2\wp'(z) = \wp'(z) +\wp'(-z)=0, \]
so they are zeros of $\wp'(z)$. Since $\wp'(z)$ has only 3 total zeros, these must be the zeros, and each one is a simple zero. 
Let
\[e_1 = \wp\left(\frac{\omega_1}{2}\right), \quad e_2 = \wp\left(\frac{\omega_2}{2}\right), \quad \text{and}\quad e_3 = \wp\left(\frac{\omega_1+\omega_2}{2}\right).\]
\begin{prop}\label{prop: simple roots}
    If $w\in\{e_1,e_2,e_3\}$ then the function $\wp(z)-w$ has a single double zero. Otherwise, $\wp(z) -w$ has two simple zeros.
\end{prop}
\begin{proof}
    Since $\wp(z) - w$ has a single pole which has order 2, it must have two zeros. Hence it has either a double zero or two simple zeros. The first possibility happens if and only if the solution $\wp(z) - w$ satisfies $\wp'(z)=0$. This is equivalent to $w\in \{e_1,e_2,e_3\}$, by the discussion above. 
\end{proof}
Using arguments like this, we can construct many elliptic functions.
\begin{leftbar}
\begin{eg}[An elliptic function with two simple poles]
    Choose any $w\not\in \{e_1,e_2,e_3\}$. Then the elliptic function \[\frac{1}{\wp(z)-w}\] has two poles, both of which are simple.
\end{eg}
\end{leftbar}

\begin{leftbar}
\begin{eg}[$\wp$ for different lattices]\label{ex: wp for lattices}
    Technically, the Weierstrass function $\wp(z)$ depends on the choice of the lattice $\Lambda$, so we should really write $\wp_{\Lambda}(z)$. In practice, we don't do this, since the lattice is clear from context. However, there are interesting relations between $\wp_{\Lambda}(z)$ for different choices of $\Lambda$. We give two examples below.

    \begin{enumerate}[leftmargin=*,label=$(\alph*)$]
        \item For $a\in\mathbb C\setminus \{0\}$, write $a\Lambda=\{a\omega :\omega\in\Lambda\}$. Then we have
\begin{align*}
    \wp_{a\Lambda}(z) &= \frac{1}{z^2} + \sum_{\omega\in \Lambda^*} \frac{1}{(z+a\omega)^2}-\frac{1}{a^2\omega^2}\\
    &=\frac{1}{a^2}\left(\frac{a^2}{z^2} + \sum_{\omega\in \Lambda^*} \frac{1}{(z/a+ \omega)^2}-\frac{1}{\omega^2}\right)\\
    &= \frac{\wp_{\Lambda}(z/a)}{a^2}.
\end{align*}
\item For convenience of notation, let $\wp = \wp_\Lambda$. Consider the function 
\[f(z) = \wp(z)+\wp(z+\omega_1/2) + \wp (z+\omega_1/2) + \wp(z+(\omega_1+\omega_2)/2).\]
Then $f$ is an elliptic function on $\Lambda/2$, with a double pole at each point in $\Lambda/2$, and no poles anywhere else. Furthermore, the $z^{-2}$ coefficient in the Laurent expansion of $f(z)$ is $1$. Hence, $f(z) - \wp_{\Lambda/2}(z)$ has no poles, and hence is constant. To compute this constant, we look at the constant term in the Laurent expansion of $f(z)$. The Laurent series for $\wp(z)$ tells us that $\wp(z)$ has constant term equal to zero. Hence, the value equals $e_1+e_2+e_3$. Since $\wp_{\Lambda/2}(z)$ also has constant term equal to zero, it follows that
\begin{align*}
\wp_{\Lambda/2}(z) =\, &\wp(z)+\wp(z+\tfrac{\omega_1}{2})+ \wp (z+\tfrac{\omega_2}2) + \wp(z+\tfrac 12(\omega_1+\omega_2))\\ &- e_1 -e_2 - e_3.
\end{align*}
% As we shall see below, $e_1+e_2 +e_3 = 0$
    \end{enumerate}
\end{eg}
\end{leftbar}

\medskip
The example above are manifestations of a much greater phenomenon, which is the following remarkable result.
\begin{theorem}\label{thm: classification}
    Every elliptic function can be expressed as a rational function in $\wp(z)$ and its derivative $\wp'(z)$. 
\end{theorem}
This theorem implies that $\wp(z)$ is the single ``most important'' elliptic function, and that studying it amounts to studying all elliptic functions.

The result will be a consequence of the following lemma.
\begin{lemma}\label{lem: even classification}
    Every \textit{even} elliptic function can be expressed as a rational function in $\wp(z)$.
\end{lemma}
\begin{proof}
Let $f(z)$ be an even elliptic function, and let
\[f(z) = \sum_{n=-\infty}^\infty a_nz^n\]
be the Laurent expansion of $f$. Then $a_n =0$ for odd $n$. The order of the pole at $0$ is thus even (this order could be negative or zero). Hence, there exists a unique integer $m$ such that $f(z)\wp(z)^m$ has neither a zero nor pole at $z=0$. 

Replacing $f(z)$ with $f(z)\wp(z)^m$, we may assume that $f$ has neither a zero nor pole at $z=0$. Now since $f$ is even, if $a$ is a root of $f$, then so is $-a$. Thus the multiset of roots of $f$ can be partitioned into $\{a,-a\}$. Restricting to the fundamental parallelogram, let $a_1,\dots, a_k, -a_1,\dots, -a_k$ be the zeros of $f$ (technically, we should write $-a_1\pmod{\Lambda}$, but this makes the notation clunky). Hence, by \cref{prop: simple roots}, the function
\[(\wp(z) - \wp(a_1)) \cdots (\wp(z) - \wp(a_k))\]
has the same multiset of zeros as $f$. Furthermore, it has a pole of order $2k$ at $0$.

Similarly, let $b_1,\dots, b_l, -b_1,\dots, b_k$ be the poles of $f$ (observe that $k = \ell$ by \cref{thm: count zeros and poles}). The function 
\[\frac{1}{(\wp(z) - \wp(b_1)) \cdots (\wp(z) - \wp(b_k))}\]
has the same set of poles as $f$, and it has a zero of order $2k$ at $0$. Thus, the function
\[\frac{(\wp(z) - \wp(a_1)) \cdots (\wp(z) - \wp(a_k))}{(\wp(z) - \wp(b_1)) \cdots (\wp(z) - \wp(b_k))}\]
has the same set of roots and zeros as $f$ (since the zero and pole at $z=0$ cancel). It follows that this function must in fact equal a multiple of $f$, and we conclude.
\end{proof}
Proving the original theorem is now straightforward.
\begin{proof}[Proof of \cref{thm: classification}]
    Every elliptic function $f$ can be written as
    \[f = f_e +f_o\]
    for even and odd elliptic functions $f_e$ and $f_o$, respectively. Concretely, we may set
    \[f_e(z) = \tfrac{1}{2}(f(z) + f(-z)) \quad \text{and}\quad f_o(z) = \tfrac{1}{2}(f(z) - f(-z)).\]
    By \cref{lem: even classification}, $f_e$ is a rational function in $\wp(z)$. Furthermore, $\wp'(z)$ is an odd function, so $f_o/\wp'(z)$ is an even elliptic function. Hence it is a rational function in $\wp(z)$ too. Multiplying this by $\wp'(z)$, we conclude.
\end{proof}
Using this theorem we can conclude, for example, that $\wp'(z)^2$ is in fact a rational function in $\wp(z)$. In the next section we will derive the explicit formula.
\begin{leftbar}
\begin{eg}
    We need to a little careful when using this theorem. For example, it seems like the function $e^{\wp(z)}$ must be a rational function of $\wp(z)$. But of course this can't be true, since that would imply that $e^w$ is a rational function in $w$. What's going on here? While $e^{\wp(z)}$ \textit{is} periodic with respect to $\Lambda$, it is \textit{not} a meromorphic function. Indeed, the double pole of $\wp(z)$ at $0$ induces a pole of \textit{infinite} order at $0$. 
\end{eg}
\end{leftbar}

\section{Eisenstein Series}
Like many other functions in complex analysis, there are numerous interesting ways to represent the Weierstrass function. Most notably, we have a Laurent series expansion.
\begin{theorem}
    We have
    \begin{align*}
        \wp(z) = \frac{1}{z^2}+\sum_{k=1}^\infty(2k+1)E_{2k+2}z^{2k},
    \end{align*}
    where
    \[E_k = \sum_{\omega\in \Lambda_{\ne 0}}\frac{1}{\omega^k}\]
    are called the \emph{Eisenstein series}.
\end{theorem}
\begin{proof}
    We just re-write \cref{defn: wp} as a Laurent series in $z$:
    \begin{align*}
        \wp(z) &= \frac{1}{z^2}+\sum_{\omega\in \Lambda_{\ne 0}} \frac{1}{(z+\omega)^2}-\frac{1}{\omega^2}\\
        &= \frac{1}{z^2}+\sum_{\omega\in \Lambda_{\ne 0}}\frac{1}{(\omega-z)^2}-\frac{1}{\omega^2}\\
        &= \frac{1}{z^2}+\sum_{\omega\in \Lambda_{\ne 0}}\frac{1/\omega^2}{(1-z/\omega)^2}-\frac{1}{\omega^2}\\
        &= \frac{1}{z^2}+\sum_{\omega\in \Lambda_{\ne 0}}\frac{1}{\omega^2}\sum_{k=1}^\infty(k+1) \left(\frac{z}{\omega}\right)^{k }- \frac{1}{\omega^2}\\
        &= \frac{1}{z^2} + \sum_{\omega\in \Lambda_{\ne 0}}\sum_{k=1}^\infty (k+1)\frac{z^k}{\omega^{k+2}}\\
        &= \frac{1}{z^2} + \sum_{k=1}^\infty(k+1)z^k\sum_{\omega\in \Lambda_{\ne 0}} \frac{1}{\omega^{k+2}}\\
        &= \frac{1}{z^2} + \sum_{k=1}^\infty(k+1)E_{k+2}z^k.
    \end{align*}
    Since $E_{n}=0$ for odd $n$ (the $\omega$ term and the $-\omega$ term cancel), the theorem is proved.
\end{proof}
This theorem allows us to say some interesting things about $\wp$. For example, we  obtain the following differential equation.
\begin{prop}
    We have
    \[(\wp')^2 = 4\wp^3 - 60 E_4 \wp - 140E_6.\]
\end{prop}
\begin{proof}
    Using the previous theorem, we first compute
    \begin{align*}
            \wp ' (z) &= \frac{-2}{z^3} + 6E_4z + 20E_6z^3 + \cdots,\\
            (\wp ' (z))^2 &= \frac{4}{z^6} - \frac{24E_4}{z^2} -80E_6 + \cdots,\\
            \wp(z)^3 &= \frac{1}{z^6} + \frac{9E_4}{z^2} + 15E_6 + \cdots.
    \end{align*}
    Then simple algebra tells us that $(\wp')^2 -4\wp^3 +60 E_4 \wp +140E_6$ has no poles: indeed, the only pole could be at $0$, and we can manually see that this particular combination of Laurent series kills all the $z^{-k}$ terms. Since the value at $z = 0$ is $0$, it follows that the function is fact $0$, completing the proof.
\end{proof}
This particular result is very useful for understanding the Weierstrass $\wp$-function. For example, we can use it to derive a beautiful addition formula for $\wp$.

\begin{theorem}\label{thm: addition formula}
    We have
    \[\wp(w+z) = \frac{1}{4}\left(\frac{\wp'(z)-\wp'(w)}{\wp(z) - \wp(w)}\right)^2 - \wp(z) - \wp(w).\]
\end{theorem}
The proof we are about to present might seem extraordinarily clever, but we will explain the motivation after.

We first need the following result.
\begin{lemma}\label{lem: group law}
    The points
    \[P = (\wp (z), \wp'(z)), \quad Q = (\wp(w),\wp'(w)), \quad R = (\wp(z+w), - \wp'(z+w))\]
    are collinear $($in $\mathbb C^2)$.
\end{lemma}
\begin{proof}
    Choose constants $A$ and $B$ such that
    \[\wp'(z) = A\wp(z) + B \quad \text{and}\quad \wp'(w)  = A\wp(w) +B.\]
    In other words, $y = Ax + B$ is the line through the points $P$ and $Q$. Now consider the elliptic function
    \[\wp'(\zeta )-A\wp(\zeta) - B.\]
    This has a pole of order $3$ at $0$, by our analysis of the pole structure of $\wp$. Hence, it must have three zeros; which by \cref{prop: sum of roots} must some to some element of $\Lambda$. We know that $z$ and $w$ are two of the roots, so up to congruence $-z-w$ is the last one. Hence, $(\wp(-z-w) ,\wp'(-z-w))= (\wp(z+w),-\wp(z+w))=R$ is on the line $y=Ax + B$, so we conclude. 
\end{proof}

\begin{proof}[Proof of \cref{thm: addition formula}]
    The three points $P$, $Q$, $R$ are collinear points lying on the curve $y^2 =4x^3 - 60E_4 x - 140E_6$, by the differential equation
    \[(\wp')^2 = 4\wp^3 - 60E_4 \wp - 140 E_6.\]
    The slope of the line through $P$, $Q$, and $R$ is 
    \[m = \frac{\wp'(z)-\wp'(w)}{\wp(z) - \wp(w)}.\]
    Hence, the line may be written as $y = mx + c$, for some constant $c$. It follows that all three values $x = \wp(z)$, $\wp(w)$, and $\wp(z+w)$ are solutions to the equation
    \[(mx+c)^2 = 4x^3 - 60E_4 x - 140E_6,\]
    which rearranges to
    \[4x^3 -m^2x^2 -\text{lower degree terms} = 0.\]
    This is a cubic polynomial in $x$, and hence Vieta's formulas imply that the sum of the roots equals zero:
    \[\wp(z) + \wp(w) + \wp(z+w) = \frac{m^2}{4}.\qedhere\]
\end{proof}

As we see above, the points $(\wp(z), \wp'(z))$ lie on the cubic curve \[y^2 = 4x^3 - 60E_4 x -140E_6.\] This curve is an \textit{elliptic curve}, and \cref{lem: group law} proves an extremely important property about this parametrization: \textit{the map} \[(\wp(z),\wp'(z))\mapsto z\] \textit{is compatible with the group law on this elliptic curve}.  In other words, we have that \[(\wp(z),\wp'(z)) + (\wp(w),\wp'(w')) = (\wp(z+w),\wp'(z'+w')),\]
where the addition on the left hand side denotes elliptic curve addition. As the lattice varies, this actually parametrized \textit{all} elliptic curves over $\mathbb C$. This view explains the above proof: \cref{thm: addition formula} is nothing but a consequence of this relation between $\wp$ and elliptic curves.

The Eisenstein series appearing in the previous two results are important in their own right, and we devote the remainder of the section to studying their properties.

% We conclude with some interesting properties of Eisenstein series.
\begin{prop}
    Every Eisenstein series $E_k$ appearing in the Laurent expansion of $\wp(z)$ can be expressed as a rational function in $E_4$ and $E_6$.
\end{prop}
\begin{proof}
We use the differential equation
\[(\wp')^2 = 4\wp^3 - 60E_4 \wp - 140 E_6.\tag{$*$}\]
We show that $E_{2k+2}$ can be expressed as a polynomial in $E_4,E_6,\dots, E_{2k}$ for $k\ge 3$, which is enough to imply the result.

To do this, expand $(*)$ as a Laurent series in $z$ using Theorem 5.1. We consider the $z^{2k-4}$ term in this expansion. The only terms contributing a $E_{2k+2}$ coefficient are $(\wp')^2$ and $4\wp^3$. The coefficient of the first is $-4k(2k+1)E_{2k+2}$, and that of the second is $12(2k+1)E_{2k+2}$. The remaining coefficients are some polynomial in $E_4,E_6,\dots, E_{2k}$. Hence, \[(12+4k)(2k+1)E_{2k+2}\] is a polynomial in $E_4,E_6,\dots, E_{2k}$, which implies the result.
\end{proof}
We obtain the following cute corollary.
\begin{cor}
    The function $\wp(z)$ is real for all $z\in \mathbb R$ (excluding the poles) if and only if both $E_4$ and $E_6$ are real.
\end{cor}
\begin{proof}
Suppose that $E_4$ and $E_6$ are real. Then every coefficient of the Laurent series expansion of $\wp(z)$ is a polynomial combination of $E_4$ of $E_6$, so all the coefficients are real. It follows that $\wp(z)\in \mathbb R$ for all $z\in \mathbb R$. 

Conversely, if $\wp(z)$ is real for all real $z$, then \[\wp(z)/z^2 - 1/z^4 \] is real for all nonzero $z$. As $z\to 0$, this approaches $3E_4$ (from the Laurent series of $\wp$), so by continuity $E_4$ is real. Similarly, \[\wp(z)/z^4 - 3E_4/z^2 - 1/z^6\] is real for real $z\ne 0$. As $z\to 0$, this approaches $5E_6$, so by continuity $E_6$ is real as well.
\end{proof}
Now, we specialize to the case where $\omega_1 = 1$ and $\omega_2 = \tau$ is in the upper half-plane $\mathbb H$. This actually does not lose any generality, since any lattice may be scaled to these specifications. We can then consider the Eisenstein series as functions $E_k(\tau)$ of $\tau$. 

The previous two results suggest that $E_4$ and $E_6$ are the most ``important'' Eisenstein series, and indeed they have some very special connections. We define the \textit{modular invariants} $g_2(\tau) = 60E_4(\tau)$ and $g_3(\tau) = 140E_6(\tau)$ to be the coefficients appearing in the differential equation for $\wp(z)$. As mentioned above, it turns out that elliptic curve over $\mathbb C$ can be parametrized by complex numbers $\tau$ in the upper half-plane\footnote{If $\Lambda$ is the lattice generated by $1$ and $\tau$, then the elliptic curve corresponding to $\tau$ is isomorphic to $\mathbb C/\Lambda$ as groups.}. The connection is that the \textit{j-invariant} $j(\tau)$ of this elliptic curve can be expressed as
\[j(\tau) = 1728 \frac{g_2(\tau)^3}{g_2(\tau) - 27g_3(\tau)^2}.\]
All of these functions are what are known as \textit{modular functions}: they are symmetric with respect to the $SL_2(\mathbb Z)$-action on the upper half-plane. 

% \newpage

This symmetry manifests as a self-similarity structure in the complex plots of these functions:

\vskip 5pt
\begin{center}
        \begin{minipage}{0.4\textwidth}
            \begin{center}
                \includegraphics[width = 0.9\textwidth]{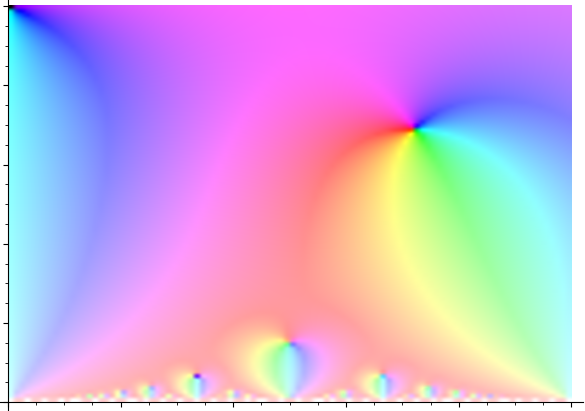}

                Plot of $g_2(\tau)$
            \end{center}
        \end{minipage}
        \begin{minipage}{0.4\textwidth}
            \begin{center}
                \includegraphics[width = 0.9\textwidth]{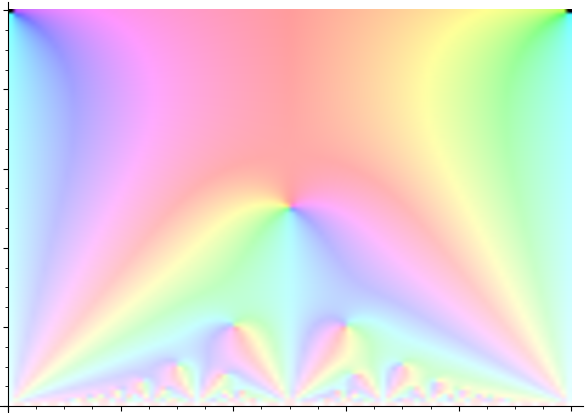}

                Plot of $g_3(\tau)$
            \end{center}
        \end{minipage}
\end{center}
\vskip 5pt
The following result proves that the Eisenstein series are modular forms.
\begin{prop}\label{prop: modular}
    Let $k\ge 4$ be even. The Eisenstein series $E_k(\tau)$ satisfy 
    \[E_k\left(\frac{a\tau+b}{c\tau + d}\right) =(c\tau + d)^k E_k(\tau)\]
    for all $\begin{pmatrix}a & b \\ c & d\end{pmatrix}\in \mathrm{SL}_2(\mathbb Z)$.
\end{prop}
\begin{proof}
    We may write
    \begin{align*}
        E_k\left(\frac{a\tau+b}{c\tau + d}\right) &= \sum_{\substack{m,n\in \mathbb Z\\ (m,n)\ne (0,0)}} \frac{1}{\left(\frac{m(c\tau + d) + n(a\tau + b))}{c\tau + d}\right)^k}\\
        &=(c\tau + d)^k\sum_{\substack{m,n\in \mathbb Z\\ (m,n)\ne (0,0)}} \frac{1}{((mc+na)\tau + (md + nb))^{k}}.\\
    \end{align*}
    Since $\begin{pmatrix}
        a & b\\ c & d
    \end{pmatrix}\in \mathrm{SL}_2(\mathbb Z)$ is invertible over the integers, and
    \[\begin{pmatrix}
        mc+na \\
        md + nb
    \end{pmatrix} = \begin{pmatrix}
        n & m
    \end{pmatrix}\begin{pmatrix}
         a & b \\
         c & d
    \end{pmatrix},\]
    it follows that $(mc+na, md+nb)$ takes each value in $\mathbb Z^2$ for exactly one value in $(m,n)\in \mathbb Z^2$. Since $(0,0)$ maps to $(0,0)$, it follows that we may rewrite the expression as
    \[(c\tau + d)^k \sum_{\substack{m,n\in \mathbb Z\\(m,n)\ne (0,0)}} \frac{1}{(m\tau + n)^k}  = (c\tau + d)^k E_k(\tau).\qedhere\]
\end{proof}
The special linear group $\mathrm{SL}_2(\mathbb Z)$ acts on the upper half-plane $\mathbb H$ via \textit{linear fractional transformations} $z\mapsto \frac{az+ b}{cz+d}$. Hence, the above proposition tells that although not being completely invariant under the $\mathrm{SL}_2(\mathbb Z)$-action, it behaves very symmetrically. Precisely, this result means that $E_k$ \textit{is a modular function of weight $k$}. 

The modular transformation law of \cref{prop: modular} reveals that the Eisenstein series are far more than convenient coefficients appearing in the Laurent expansion of the Weierstrass $\wp$-function. They are among the simplest examples of modular forms, objects whose remarkable symmetries encode deep arithmetic information. What began as an analytic study of doubly periodic functions has therefore led naturally to the theory of elliptic curves and modularity. Through the Eisenstein series, the geometry of lattices, the analysis of elliptic functions, and the arithmetic of modular forms become manifestations of a single underlying structure.

\section*{Conclusion}

The theory of elliptic functions provides a striking example of how different areas of mathematics can illuminate one another. Beginning with the simple idea of double periodicity, we developed the basic structure theory of elliptic functions and constructed the Weierstrass $\wp$-function, which serves as a universal building block for the subject. Its Laurent expansion introduced the Eisenstein series, whose coefficients encode fundamental information about the underlying lattice.

The differential equation satisfied by $\wp$ revealed an unexpected connection with elliptic curves, while the modular transformation properties of the Eisenstein series connected the theory to modular forms and arithmetic geometry. These relationships show how analytic objects defined by infinite series can give rise to rich algebraic and geometric structures. Although the classical theory originated in the study of elliptic integrals, it continues to occupy a central place in modern mathematics, linking complex analysis, geometry, number theory, and representation theory through a common and remarkably elegant framework.

\end{document}